\documentclass[12pt]{amsart}
\usepackage{fullpage}
\allowdisplaybreaks
\usepackage{amsmath,amssymb,amsthm}
\usepackage{mathtools}
\usepackage{verbatim}
\usepackage{enumerate}
\usepackage{hyperref}
\usepackage{dsfont}
\usepackage{color}

\usepackage[margin=1in]{geometry} 
\usepackage[subnum]{cases}

\renewcommand{\bar}{\overline}
\newtheorem{theorem}{Theorem}[section]
\newtheorem{proposition}[theorem]{Proposition}
\newtheorem{lemma}[theorem]{Lemma}
\newtheorem{corollary}[theorem]{Corollary}

\theoremstyle{definition}
\newtheorem{definition}[theorem]{Definition}

\theoremstyle{remark}
\newtheorem*{remark}{Remark}

\usepackage{thmtools}
\newcommand{\nc}{\newcommand}
\nc{\R}{\mathbb R}
\nc{\C}{\mathbb C}
\nc{\F}{\mathbb F}
\nc{\Q}{\mathbb Q}
\nc{\Z}{\mathbb Z}
\nc{\N}{\mathbb N}

\nc{\cT}{\mathcal T}
\nc{\cP}{\mathcal P}
\nc{\cM}{\mathcal M}
\nc{\cC}{\mathcal C}
\nc{\cB}{\mathcal B}
\nc{\cS}{\mathcal S}
\nc{\Mod}{\operatorname{Mod}}
\nc{\Aut}{\operatorname{Aut}}
\nc{\del}{\partial}
\nc{\dd}{\mathrm{d}}
\nc{\inter}{\mathrm{o}}
\nc{\close}[1]{\overline{#1}}
\nc{\pderiv}[2]{\frac{\partial #1}{\partial #2}}
\nc{\tr}{\operatorname{tr}}
\renewcommand{\epsilon}{\varepsilon}

\title{$L^p$ Regularity of Toeplitz Operators on Generalized Hartogs Triangles}

\author{Meijke Balay}
\address[Meijke Balay]{University of California, Los Angeles, Department of Mathematics, Los Angeles, CA 90095, USA}
\email{mysatellite99@ucla.edu}

\author{Trent Neutgens}
\address[Trent Neutgens]{University of Minnesota--Twin Cities, School of Mathematics, Minneapolis, MN 55455, USA}
\email{neutg007@umn.edu}

\author{Nick Rosen}
\address[Nick Rosen]{Carleton College, Department of Mathematics, Northfield, MN 55057, USA}
\email{rosenn@carleton.edu}

\author{Nathan Wagner}
\address[Nathan Wagner]{Washington University in St. Louis, Department of Mathematics and Statistics, St. Louis, MO 63130, USA}
\email{nathanawagner@wustl.edu}

\author{Yunus E. Zeytuncu}
\address[Yunus E. Zeytuncu]{University of Michigan--Dearborn, Department of Mathematics and Statistics, Dearborn, MI 48128, USA}
\email{zeytuncu@umich.edu}

\date{}
\begin{document}

\maketitle
\begin{abstract}
We obtain $L^p$ estimates for Toeplitz operators on the generalized Hartogs triangles $\mathbb{H}_\gamma = \{(z_1,z_2) \in \C^2: |z_1|^\gamma < |z_2|<1\}$ for two classes of positive radial symbols, one a power of the distance to the origin, and the other a power of the distance to the boundary.
\end{abstract}

\section{Introduction}
\subsection{Preliminaries}
\par Let $\Omega \subset \C^n$ be a bounded domain and let $L^2(\Omega)$ denote the Hilbert space of square integrable functions equipped with the inner product 
\[\langle f, g \rangle = \int_{\Omega} f(z) \overline{g(z)} \dd V(z)\] where $\dd V$ is the Lebesgue measure. Denote the ring of holomorphic functions on $\Omega$ by $\mathcal{O}(\Omega)$. The \emph{Bergman space} is defined by $A^2(\Omega) = L^2(\Omega) \cap \mathcal{O}(\Omega)$. Similarly $A^p(\Omega) = L^p(\Omega) \cap \mathcal{O}(\Omega).$
$A^2(\Omega)$ is a closed subspace of $L^2(\Omega)$, so that there is the orthogonal projection $\mathbf{B}_\Omega: L^2(\Omega) \to A^2(\Omega)$ called the \emph{Bergman projection}. The Bergman projection is an integral operator
    \[\mathbf{B}_\Omega f(z) = \int_{\Omega} B(z,w) f(w) dV(w)\]
and the kernel $B(z,w)$ is called the \emph{Bergman kernel}. If $\{f_n\}$ is an orthonormal basis for $A^2(\Omega)$ then the Bergman kernel can be represented as the sum 
\[B(z,w) = \sum_{n} f_n(z) \overline{f_n(w)}.\]
For the basic theory of the Bergman projection and proof of these facts see \cite[Ch. 1]{krantz}.\\
\par Let $\phi \in L^\infty(\Omega).$ The \emph{Toeplitz operator with symbol} $\phi$ is the operator $T_\phi$ defined by
\[T_\phi f (z) = \mathbf{B}_\Omega (\phi f)(z) = \int_\Omega B(z,w) f(w)\phi(w) \dd V(w).\]
Toeplitz operators are common objects in functional analysis and have been extensively studied on different function spaces, including Bergman spaces, Hardy spaces, and the Segal-Bargmann space, see e.g. \cite{axtoeplitz,segal,peller}.\\

\par The Hartogs triangle $\{(z_1,z_2) \in \C^2: |z_1|<|z_2|<1\}$ is a pseudoconvex domain that is the source of many counterexamples in several complex variables, see \cite{hartogs}.\newpage 
The \emph{generalized Hartogs triangles} recently studied by Edholm and McNeal \cite{edneal,Edholm} are a class of pseudoconvex domains in $\mathbb{C}^2$ defined for $\gamma >0$ by
\[\mathbb{H}_\gamma = \{(z_1,z_2) \in \C^2: |z_1|^\gamma < |z_2|<1\}.\]
 The generalized Hartogs triangles exhibit the same pathological behaviour as the classical Hartogs triangle due to the singularity of the boundary, which is non-Lipschitz at the origin, with the additonal surprising dependence on the rationality or irrationality of the power $\gamma$.\\
 
\par On the generalized Hartogs triangle we denote the Bergman projection by $\mathbf{B}_\gamma$ and the Bergman kernel as $B_\gamma(z,w)$. $\mathbb{H}_\gamma$ is a Reinhardt domain, meaning $(z_1, z_2) \in \mathbb{H}_\gamma$ implies $(e^{i\theta} z_1, e^{i\theta}z_2) \in \mathbb{H}_\gamma$ for all $\theta \in \R$. Using the Laurent expansion for a holomorphic function, an orthogonal basis for $A^2(\mathbb{H}_\gamma)$ is given by the monomials $\{z^{\alpha} = z_1^{\alpha_1}z_2^{\alpha_2}: \alpha \in \mathcal{A}_{\gamma}\}$, where $\mathcal{A}_{\gamma} = \{\alpha \in \Z^2: z_1^{\alpha_1}z_2^{\alpha_2} \in A^2(\mathbb{H}_{\gamma})\}$ is the set of multi-indices $\alpha = (\alpha_1 , \alpha_2)$ such that $||z_1^{\alpha_1}z_2^{\alpha_2}||_2< \infty$. In particular for $\gamma = \frac{m}{n}$ one can calculate in polar coordinates that \[\mathcal{A}_{m/n} = \{ \alpha \in \Z^2: \alpha_1 \ge 0, \; n\alpha_1 +m \alpha_2 \ge -m-n+1\}.\]
Therefore if $c_\alpha = ||z^\alpha||_2$ we can write the Bergman kernel as the sum
\begin{align}B_\gamma(z,w) = \sum_{\alpha \in \mathcal{A}_\gamma} \frac{z^\alpha \overline{w}^\alpha}{c_\alpha^2}. \label{monom}\end{align}
We write $A \lesssim B$ if there exists a constant $C$ such that $A \le C B$ and we write $A \approx B$ if $A \lesssim B$ and $B \lesssim A$.
\subsection{Main Results}
\par In general,
the regularity of $\mathbf{B}_\Omega$ depends closely on the geometry of $\Omega$. For various geometric conditions on $\Omega$, understanding the range of $p$ for which $\mathbf{B}_\Omega$ is $L^p$ bounded is an active area of research, see \cite{survey} for a comprehensive survey. If $\Omega$ is smooth and strongly pseudoconvex, or smooth, convex, and finite type then $\mathbf{B}_\Omega$ is bounded from $L^p(\Omega)$ to $A^p(\Omega)$ is bounded for all $1<p<\infty$ \cite{stein, mcneal94}. At the other extreme, in \cite{yunus} the author constructs psuedoconvex domains in $\C^2$ where $\mathbf{B}_\Omega$ is bounded if and only if $p=2$. The generalized Hartogs triangles $\mathbb{H}_\gamma$ provide another example of restricted boundedness. In \cite{edneal} it is shown that for $\gamma = \frac{m}{n}$ an irreducible fraction the Bergman projection is bounded on $L^p(\mathbb{H}_\gamma)$ if and only if $\frac{2m+2n}{m+n+1} < p < \frac{2m+2n}{m+n-1}$. Moreover if $\gamma$ is irrational the Bergman projection is bounded if and only if $p =2$.\\
\par In this paper we investigate the regularity of Toeplitz operators on $\mathbb{H}_\gamma$. We consider two classes of Toeplitz operators and show that the $L^p$ boundedness range can only increase the boundedness range for the Bergman projection in the lower range $p<2.$ The first set of symbols $\phi(z_1,z_2) = |z_2|^{\alpha}$ essentially measure the distance to the singularity at the origin. For these symbols, we have the following statement.
\begin{theorem}\label{ToeplitzA}
    Let $T_{\alpha}$ denote the Toeplitz Operator with the symbol $\phi(z_1,z_2) = |z_2|^\alpha$ where $0 \le \alpha \le \frac{m+n-1}{m}$. Then $T_\alpha$ maps $L^p(\mathbb{H}_{m/n})$ to $A^p(\mathbb{H}_{m/n})$ boundedly if and only if\\ $\frac{2m+2n}{m+n+1+m\alpha} < p <\frac{2m+2n}{m+n-1}.$
\end{theorem}
Thus increasing the power $\alpha$ improves the lower range, but not the upper range. This behaviour occurs more generally: if $\phi: \mathbb{H}_{m/n} \to \mathbb{R}$ is any positive bounded radial symbol, i.e. $\phi(z_1,z_2) = \phi(|z_1|,|z_2|)$, then $T_\phi$ does not map $L^p(\mathbb{H}_{m/n})$ to $A^p(\mathbb{H}_{m/n})$ for any $p \ge \frac{2m+2n}{m+n-1}$. For $\mathbb{H}_\gamma$ with $\gamma$ irrational, the same statement holds for all $p>2$, see Corollary \ref{radialthresh} below.

\par We mention that on the regular Hartogs triangle $\mathbb{H}$, a related phenomena was noted in \cite{yunushartog}. The authors showed that for $1<p\le \frac{4}{3}$, the Bergman projection does not map $L^p(\mathbb{H})$ to the weighted space $A^p(\mathbb{H},\omega)$ where $\omega$ is any positive continuous function of $|z_2|$.
The second set of symbols are powers of the distance to the boundary function. In \cite{cucneal} it was shown that if $\Omega \subset \C^n$ is a smoothly bounded strongly pseudoconvex domain, then distance to the boundary symbols had a smoothing effect. More specifically, if $\delta: \Omega \to \R$ is the distance to the boundary function then $T_\eta := T_{\delta^\eta}$ maps $L^p(\Omega)$ to $L^{p+G}(\Omega)$ for some $G\ge 0$ depending on $\eta \ge 0$. Moreover if $\eta$ is large enough, $T_\eta$ maps $L^p(\Omega)$ to $L^r(\Omega)$ for any $ p \le r < \infty$. We have an analogous statement on $\mathbb{H}_{m/n}$. The main difference is that the smoothing is limited by the upper threshold $p = \frac{2m+2n}{m+n-1}$.
\begin{theorem}\label{ToeplitzN}
Let $\eta\geq 0$, $M = \frac{2m+2n}{m+n-1}$, $L = \max\left(\frac{2m+2n}{m+n+1+mn\eta},1\right)$, and $p \in (L, M)$.\\

Consider the Toeplitz operator $T_\eta = T_\phi$ where 
\[\phi(z) = (|z_2|^{n} - |z_1|^{m})^\eta(1-|z_2|^2)^\eta. \] 
Then the following holds:
\begin{enumerate}[(a)]
    \item If $\eta \ge 2\left( \dfrac{1}{p} - \dfrac{1}{M}\right)$, then $T_\eta : L^p(\mathbb{H}_{m/n}) \to L^r(\mathbb{H}_{m/n})$ is a bounded operator for any $r<M$.
    \item If $\eta < 2\left( \dfrac{1}{p} - \dfrac{1}{M}\right)$ then $T_\eta: L^p(\mathbb{H}_{m/n}) \to L^{p+G}(\mathbb{H}_{m/n})$ is a bounded operator, where \[G = \dfrac{p^2}{\frac{2}{\eta} -p}.\]
\end{enumerate}
In this case $p+G < M.$\\
Furthermore, if $p\in(1,L]\cup[M,\infty)$, then $T_\eta$ does not map $L^p(\mathbb{H}_{m/n})$ to $L^r(\mathbb{H}_{m/n})$ boundedly for $r\ge p$.
\end{theorem}

\section{Proof of Theorems}
The $L^p$ regularity of the Bergman Projection on $\mathbb{H}_{m/n}$ relies on kernel estimates obtained in \cite{Edholm,edneal}. The rationality of $\gamma = \frac{m}{n}$ allows one to split the Bergman kernel into $m$ subkernels that can be explicitly computed. This gives us the crucial estimate
\begin{align}\label{kernelest}
    |B_{m/n}(z,w)| \lesssim \frac{|z_2 \bar{w}_2|^{2n -1 + \frac{1-n}{m}}}{|1-z_2 \bar{w}_2|^2|z_2^n\bar{w}_2^n - z_1^m\bar{w}_1^m|^2}.
\end{align}

\begin{definition}
Let $\mathcal{K}$ be an integral operator on $\mathbb{H}_{m/n}$ with kernel $K(z,w)$, i.e. 
\[\mathcal{K}f(z) = \int_{\mathbb{H}_{m/n}}K(z,w) f(w) \dd V(w).\]
We say $\mathcal{K}$ is of type-$(c,d)$ if its kernel satisfies 
\[|K(z_1,z_2,w_1,w_2)| \lesssim \frac{|z_2|^c|\overline{w}_2|^d}{|1-z_2\bar{w}_2|^2|z_2^n\bar{w}_2^n - z_1^m \bar{w}_1^m|^2}\]
uniformly for $(z,w) \in \mathbb{H}_{m/n} \times \mathbb{H}_{m/n}$. This generalizes the ``type-$A$'' operator defined in \cite{edneal}.
\end{definition}

\begin{proposition}\label{cd}
If $\mathcal{K}$ is a type-$(c,d)$ operator, then $\mathcal{K}:L^p(\mathbb{H}_{m/n}) \to L^p(\mathbb{H}_{m/n})$ is bounded if 
\[\dfrac{2n+2m}{dm + 2n+2m -2nm} < p < \dfrac{2n+2m}{2nm - cm}\]
provided the denominators are positive and $dm+2n+2m-2nm > 2nm -cm.$
\end{proposition}
 For the proof of Theorem \ref{ToeplitzA} and \ref{ToeplitzN} we use two different variations of Schur's Lemma. Schur's lemma is a common tool to prove boundedness of an integral operator, see e.g. \cite[Theorem 3.6]{zhu}.
\begin{lemma}\label{schurlemma}
    Let $\Omega \subset \C^n$ be a domain and $K$ a positive measurable kernel on $\Omega \times \Omega$. Let $\mathcal{K}$ be the operator with kernel $K$ and suppose that there is a positive auxiliary function $h: \Omega \to \R$ and pairs of numbers $0<a<b$, $0<a'<b'$ such that 
    \begin{align}\mathcal{K}(h^{-\epsilon})(z) &\lesssim h(z)^{- \epsilon} \qquad \text{for all } \epsilon \in (a,b) \label{schur1} \\
    \mathcal{K}(h^{-\epsilon})(w) &\lesssim h(w)^{- \epsilon} \qquad \text{for all } \epsilon \in (a',b') \label{schur2}.
\end{align}
Then $\mathcal{K}:L^p(\Omega) \to L^p(\Omega)$ is bounded for $\dfrac{a'+b}{b} < p < \dfrac{b'+a}{a}.$
\end{lemma}
\begin{proof}
Let $q$ be such that $\frac{1}{p} + \frac{1}{q} = 1$. Fix $f \in L^p(\Omega)$ and $s>0$ to be determined. 
\begin{align*}
    \left| \int_\Omega K(z,w) f(w) dV(w)\right|^p & \le  \left( \int_\Omega K(z,w) h(w)^{-s} dV(w) \right)^{p/q} \left( \int_\Omega K(z,w) |f(w)|^p h(w)^{s p/q} dV(w)\right) \\
    &\lesssim h(z)^{-s p/q} \int_\Omega K(z,w) |f(w)|^p h(w)^{sp/q} dV(w).
\end{align*}
Therefore we get 
\begin{align*}
    ||\mathcal{K}f||_p^p & \lesssim \int_\Omega \int_\Omega K(z,w) |f(w)|^p h(w)^{s p/q} h(z)^{-s p/q} dV(w) dV(z) \\
    &= \int_\Omega \left( \int_\Omega K(z,w)h(z)^{-s p/q} dV(z) \right) h(w)^{s p/q} |f(w)|^p dV(w)\\
    & \lesssim \int_\Omega h(w)^{- s p/q}h(w)^{s p/q} |f(w)|^p dV(w)\\
    & = ||f||_p^p.
\end{align*}
This calculation is valid whenever we can apply estimate \eqref{schur1} with $\epsilon = s$ and estimate \eqref{schur2} with $\epsilon = sp/q$. This occurs if $ap/q < b'$ and $bp/q > a'$, or equivalently $\frac{a'+b}{b} < p < \frac{b'+a}{a}$.
\end{proof}
The primary estimate needed to apply Lemma \ref{schurlemma} is proved in \cite[Proposition 4.4]{edneal}. The integral is related to the Bergman projection of $(1-|z|^2)^{-\epsilon}$ on the disc $\mathbb{D}$.
\begin{lemma}\label{discint}
Let $\epsilon \in (0,1)$ and $\beta \in (-\infty,2)$. Then for $z \in \mathbb{D}$,
\[\int_{\mathbb{D}}\frac{(1-|w|^2)^{-\epsilon}}{|1-z \bar{w}|^2 } |w|^{-\beta}\dd V(w)\lesssim (1-|z|^2)^{-\epsilon}.\]
\end{lemma}
Now, we can begin the proof of Proposition \ref{cd} in earnest. 
\begin{proof}[Proof of Proposition \ref{cd}]
We verify the estimates in Schur's lemma with the auxiliary function \[h(z) = (|z_2|^{2n} - |z_1|^{2m})(1-|z_2|^2).\]
The proof is the same as \cite[Proposition 4.2]{edneal} except for the exponents.\\

Let $\mathcal{K}$ be an operator of type-$(c, d)$ on $\mathbb{H}_{m/n}$ and assume that $dm+2n+2m-2mn>2mn-cm>0$.\\

For some $\epsilon$ to be determined,
\begin{align*}
    \mathcal{K}(h^{-\epsilon})(z)&\lesssim\int_{\mathbb{H}_{m/n}}\frac{|z_2|^c|w_2|^d\left(|w_2|^{2n}-|w_1|^{2m}\right)^{-\epsilon}\left(1-|w_2|^2\right)^{-\epsilon}}{|1-z_2\bar{w_2}|^2|z_2^n\bar{w_2}^n-z_1^m\bar{w_1}^m|^2}\dd V(w)\\
    &=\int_{\mathbb{D}^*}\frac{|z_2|^c|w_2|^d\left(1-|w_2|^2\right)^{-\epsilon}}{|1-z_2\bar{w_2}|^2}\\
    &\:\:\:\:\:\times\left[\int_W \frac{\left(|w_2|^{2n}-|w_1|^{2m}\right)^{-\epsilon}}{|z_2^n\bar{w_2}^n-z_1^m\bar{w_1}^m|^2}\dd V(w_1)\right]\dd V(w_2)
\end{align*}
where $\mathbb{D}^* = \{w_2:0<|w_2|<1\}$ and $W = \{w_1:|w_1|<|w_2|^{n/m}\}$ with $w_2$ fixed. Denoting the integral in brackets by $I$, we have
\begin{equation*}
    I = \frac{1}{|z_2|^{2n}|w_2|^{2n+2n\epsilon}}\int_W\left(1-\left|\frac{w_1^m}{w_2^n}\right|^2\right)^{-\epsilon}\left|1-\frac{z_1^m\bar{w_1}^m}{z_2^n\bar{w_2}^n}\right|^{-2}\dd V(w_1).
\end{equation*}
We make the substitution $u=\frac{w_1^m}{w_2^n}$. Lemma \ref{discint} yields
\begin{align*}
    I&=\frac{|w_2|^{2n/m-2n-2n\epsilon}}{m|z_2|^{2n}}\int_{\mathbb{D}}\frac{(1-|u|^2)^{-\epsilon}}{\left|1-z_1^mz_2^{-n}\bar{u}\right|^2}|u|^{2/m-2}\dd V(u)\\
    &\lesssim \frac{|w_2|^{2n/m-2n-2n\epsilon}}{|z_2|^{2n}}\left(1-\left|\frac{z_1^m}{z_2^n}\right|^2\right)^{-\epsilon}\\
    &= \frac{|w_2|^{2n/m-2n-2n\epsilon}}{|z_2|^{2n-2n\epsilon}}\left(|z_2|^{2n}-|z_1|^{2m}\right)^{-\epsilon}.\\
\end{align*}
Thus, we have
\begin{align}
    \mathcal{K}(h^{-\epsilon})(z) &\lesssim |z_2|^{c+2n\epsilon -2n} (|z_2|^{2n} - |z_1|^{2m})^{-\epsilon} \int_{\mathbb{D}^*}  \dfrac{(1-|w_2|^2)^{-\epsilon}}{|1-z_2\bar{w}_2|^2} |w_2|^{d+2n/m -2n -2n\epsilon} \dd V(w_2).\label{est1}
\end{align}
To estimate the integral with Lemma \ref{discint} we need $d + 2n/m -2n -2n\epsilon > -2.$ Thus 
\[\epsilon < \frac{1}{2n}\left( d + \frac{2n}{m} - 2n+2\right).\]
Assuming this holds, \eqref{est1} becomes
\begin{align*}
    \mathcal{K}(h^{-\epsilon})(z) &\lesssim |z_2|^{c+2n\epsilon -2n}(|z_2|^{2n} - |z_1|^{2m})^{-\epsilon}(1-|z_2|^2)^{-\epsilon}\\
    &= |z_2|^{c+2n\epsilon-2n} h(z)^{-\epsilon}.
\end{align*}
To get a uniform bound we need $c+2n\epsilon -2n \ge 0$, so $\epsilon \ge 1-c/2n$. Thus we may take $a = 1-c/2n$ and $b = \frac{1}{2n}(d+\frac{2n}{m} -2n+2)$.\\

The second estimate \eqref{schur2} is almost identical, except $z$ and $w$ are switched. We get the same range of $\epsilon$ but with $c$ and $d$ swapped, i.e.
\[1-\frac{d}{2n} \le \epsilon < \frac{1}{2n} \left( c + \frac{2n}{m} - 2n +2 \right).\]
Taking $a' = 1-d/2n$ and $b' = \frac{1}{2n}(c +\frac{2n}{m} -2n +2)$ and plugging this into Lemma \ref{discint} we get the desired range of $p$.
\end{proof}

\subsection{Proof of Theorem \ref{ToeplitzA}}
Since the Bergman projection is a type-$(A,A)$ operator where $A = 2n - 1 + \frac{1-n}{m}$ it follows that $T_\alpha$ is a type-$(A,A+\alpha)$ operator. Proposition \ref{cd} then shows that $T_\alpha$ is $L^p$ bounded for $\frac{2m+2n}{m+n+1+m\alpha} <p<\frac{2m+2n}{m+n-1}.$\\

The proof of unboundedness is based on the following two lemmas. The first is standard, see \cite[Theorem 1.9]{zhu}.
\begin{lemma}\label{adjointLemma}
Let $\Omega \subset \C^n$ be a bounded domain and $p \ge 1$. If $T:L^p(\Omega) \to L^p(\Omega)$ is a bounded linear operator, then its adjoint $T^*$ is a bounded linear operator $T^*:L^q(\Omega) \to L^q(\Omega)$ where $\frac{1}{p}+\frac{1}{q} =1$ with the convention that $q = \infty$ if $p=1$.
\end{lemma}

\begin{remark}Note that if $T_\phi$ is a Toeplitz operator on a domain $\Omega$, the adjoint is given by $T_\phi^* f = \overline{\phi} \mathbf{B}_\Omega(f)$ by self-adjointness of the Bergman projection.
\end{remark}

\begin{lemma}\label{RadialToeplitzPolynomial}
Let $\phi:\mathbb{H}_\gamma \to \R$ be any bounded radial function and suppose $(\beta_1,\beta_2) \in \mathcal{A}_\gamma$ and $(\beta_1, -\beta_2) \in \mathcal{A}_\gamma$. Then 
\[T_\phi(z_1^{\beta_1}\bar{z}_2^{\beta_2}) = Cz_1^{\beta_1}z_2^{-\beta_2}\]
for some constant $C$.
\end{lemma}
\begin{proof}
Calculate using polar coordinates $w_j = r_j e^{i \theta_j}$, $j=1,2$. Set $H = \{(r_1,r_2): r_1^\gamma < r_2 <1\}$.
\begin{align*}
    T_\phi(z_1^{\beta_1}\bar{z}_2^{\beta_2}) &= \int_{\mathbb{H}_\gamma} B_\gamma(z,w) \phi(w)w_1^{\beta_1}\bar{w}_2^{\beta_2}dV(w)\\
    &= \int_{\mathbb{H}_\gamma} \sum_{\alpha \in A_\gamma} \frac{z^\alpha \bar{w}^\alpha}{c_\alpha^2} \phi(w)w_1^{\beta_1}\bar{w}_2^{\beta_2}dV(w) \\
    &=\sum_{\alpha \in A_\gamma} \frac{z^\alpha}{c_\alpha^2}\int_{\mathbb{H}_\gamma}r_1^{\alpha_1+ \beta_1+1}e^{i\theta_1(\beta_1-\alpha_1)}r_2^{\alpha_2+\beta_2+1}\phi(r_1,r_2)e^{-i\theta_2(\beta_2+\alpha_2)}dr \;d\theta\\
    &= \sum_{\alpha \in A_\gamma} \frac{z^\alpha}{c_\alpha^2} \left( \int_0^{2\pi} e^{i \theta_1 (\beta_1-\alpha_1)}d\theta_1 \right) \left( \int_0^{2\pi} e^{-i \theta_2(\beta_2+\alpha_2)} d\theta_2 \right)\left(\int_H \phi(r_1,r_2)r_1^{\alpha_1+\beta_1+1}r_2^{\alpha_2+\beta_2+1}dr\right)\\
    &=Cz_1^{\beta_1}z_2^{-\beta_2}.
\end{align*}
\end{proof}
\begin{remark}
For $\phi$ and $\beta$ as in the Lemma, we also have
 \[T_\phi^*(z_1^{\beta_1}\overline{z_2}^{\beta_2}) = \phi \mathbf{B}_\gamma (z_1^{\beta_2}\overline{z_2}^{\beta_2}) = C \phi \cdot  z_1^{\beta_1}z_2^{-\beta_2}.\]
\end{remark}

\begin{proof}[Proof of Unboundedness]
\par To prove unboundedness for $1 \leq q \leq \frac{2m+2n}{m+n+1+m\alpha}$ we claim that the adjoint $T_\alpha^*$ is unbounded for the conjugate range $\frac{2m+2n}{m+n-1-m\alpha} \le p \le \infty$. Then the contrapositive of Lemma \ref{adjointLemma} will finish the proof.\\
\par To prove the claim, we exhibit a bounded monomial function $f$ such that $T_\alpha^*(f) \notin L^p(\mathbb{H}_{m/n})$ for any $ \frac{2m+2n}{m+n-1-m\alpha} \le p \le \infty$. If $\beta_1 \ge 0$ is an integer, let $\ell(\beta_1)$ be the least integer such that $(\beta_1, \ell(\beta_1)) \in \mathcal{A}_{m/n}$. If $0 \le j \le m-1$ is the residue of $\beta_1$ modulo $m$, this is given by 
\begin{align}\ell(\beta_1) &= -1-\frac{n(\beta_1 - j)}{m} - \frac{(j+1)n -1}{m} . \label{leastexp}\end{align} Since $\gcd(n,m)=1$ there exists $\beta_1 \geq 0$ such that $n(\beta_1 +1) \equiv 1 \pmod{m}$. Let $\beta_2 = -\ell(\beta_1) = 1+\frac{n(\beta_1+1)-1}{m}$. Then $f(z) = z_1^{\beta_1}\bar{z_2}^{\beta_2}$ is bounded since $\beta_2>0$. On the other hand by Lemma \ref{RadialToeplitzPolynomial} 
\begin{align*}
    ||T_\alpha^*(z_1^{\beta_1}\overline{z_2}^{\beta_2})||_p^p &= \int_{\mathbb{H}_{m/n}}\big{|}|{z_2}|^{\alpha}C z_1^{\beta_1}z_2^{-\beta_2}\big{|}^p \, \dd V(z)\\
    &\approx\int_0^1 \int_0^{r_2^{n/m}}r_1^{p\beta_1 + 1}r_2^{p(\alpha - \beta_2)+1} \, \dd r_1 \, \dd r_2\\
    &\approx \int_0^1 r_2^{\frac{n}{m}(p\beta_1+2) + p\alpha - p\beta_2+1} \, \dd r_2
\end{align*}
which diverges exactly when
\[
\frac{n}{m}(p\beta_1+2) + p\alpha - p\beta_2+1 \le -1.
\]
Plugging in $\beta_1$ and $\beta_2 = -\ell(\beta_2) = 1+\frac{n(\beta_1+1)-1}{m}$ this is equivalent to 
\begin{align*}
    p\left( \frac{1-n}{m} - 1 + \alpha \right) & \le -2 -\frac{2n}{m}.
\end{align*}
The expression in parenthesis is negative if $ \alpha < \frac{m+n-1}{m}$ so we get 
\[
p \geq \frac{2n+2m}{m+n-1-m\alpha}.
\]
Therefore $T_\alpha^*(f) \notin L^p(\mathbb{H}_{m/n})$. Moreover if $\alpha = \frac{m+n-1}{m}$ then $\frac{2n+2m}{m+n-1-m\alpha} = \infty$, and it is clear that $T_\eta^*f \notin L^\infty(\mathbb{H}_{m/n})$. This proves the claim.\\

\par Since $|\mathbf{B}_{m/n}f (z)| \approx |T_\alpha f (z)|$, the same calculation with $\alpha=0$ shows that $||T_\alpha(f)||_q < \infty$ only if $q \ge \frac{2n+2m}{m+n-1}$. We conclude that $T_\alpha$ maps $L^p(\mathbb{H}_{m/n})$ to $A^p(\mathbb{H}_{m/n})$ boundedly if and only if $ \frac{2m+2n}{m+n+1+m\alpha} < p < \frac{2m+2n}{m+n-1}$.
\end{proof}
We note that with Lemma \ref{RadialToeplitzPolynomial}, the same argument for unboundedness applies for any positive bounded radial symbol $\phi$. For irrational $\gamma$, the argument in Section 6 of \cite{edneal} applies verbatim. We summarize these observation in the following corollary.
\newpage 

\begin{corollary}\label{radialthresh}
Let $\phi: \mathbb{H}_{\gamma} \to \R$ be a positive bounded radial function. 
\begin{enumerate}[(a)]
    \item If $\gamma$ is rational, then $T_\phi$ does not map $L^p(\mathbb{H}_\gamma)$ to $A^p(\mathbb{H}_\gamma)$ for any $p \ge \frac{2m+2n}{m+n-1}$.
    \item  If $\gamma$ is irrational,  then $T_\phi$ does not map $L^p(\mathbb{H}_\gamma)$ to $A^p(\mathbb{H}_\gamma)$ for any $p>2$.
\end{enumerate}
\end{corollary}
\begin{remark} In case (b), if $\phi$ is bounded away from zero near the boundary then the result also holds for $p<2$. In general $T_\phi$ may be bounded for $p<2$, for example if $\phi$ is compactly supported. We do not know the exact boundedness range for $p<2$ in general because we do not have estimates for the Bergman kernel in the irrational case. 
\end{remark}
\subsection{Proof of Theorem 2.2}
To prove boundedness of $T_\eta$ we will use a different generalization of Schur's lemma.
\begin{lemma}\label{schurlemma2}
Let $\Omega \subset \C^n$ be a bounded domain, $1<p,r < \infty$, and $\frac{1}{p} + \frac{1}{q} =1$. Suppose $K$ is a kernel such that for some $0 \le t \le 1$ and auxiliary function $h$, the following estimates hold
\begin{align} 
    \int_\Omega |K(z,w)|^{tq}|h(w)|^{-\epsilon} dV(w) &\lesssim |h(z)|^{-\epsilon} \qquad \text{for all } \epsilon \in [a,b)\label{genschur1}\\
    \int_\Omega |K(z,w)|^{(1-t)r}|h(z)|^{-\epsilon} dV(z) &\lesssim |h(w)|^{-\epsilon} \qquad \text{for all } \epsilon \in [a',b') \label{genschur2}.
\end{align}
If $\dfrac{a'}{b} <\dfrac{r}{q} < \dfrac{b'}{a}$ then the integral operator $\mathcal{K}: L^p(\Omega) \to L^r(\Omega)$ with kernel $K$ is bounded.
\end{lemma}
\begin{proof}
In a similar manner to Lemma \ref{schurlemma}, let $s>0$ to be determined. Using \eqref{genschur1} we get
\begin{align*}
    |\mathcal{K}f(z)| &= \left| \int_\Omega K(z,w)^{1-t}|h(w)|^s f(w) K(z,w)^t|h(w)|^{-s} dV(w)\right|\\
    &\le \left( \int_\Omega K(z,w)^{(1-t)p}|f(w)|^p|h(w)|^{sp} dV(w)\right)^{1/p}\left(\int_\Omega K(z,w)^{tq}|h(w)|^{-q s}dV(w)\right)^{1/q}\\
    & \lesssim \left( \int_\Omega K(z,w)^{(1-t)p}|f(w)|^p|h(w)|^{sp} dV(w)\right)^{1/p} |h(z)|^{-s}.
\end{align*}
Therefore using Minkowski's integral inequality and \eqref{genschur2} 
\begin{align*}
||\mathcal{K}f||_r^p & \lesssim \left( \int_\Omega \left( \int_\Omega K(z,w)^{(1-t)p}|f(w)|^p|h(w)|^{sp} |h(z)|^{-sp} dV(w) \right)^{r/p} dV(z) \right)^{p/r}\\
& \le \int_\Omega \left( \int_\Omega K(z,w)^{(1-t)r}|f(w)|^r|h(w)|^{sr} |h(z)|^{-sr}dV(z) \right)^{p/r} dV(w)\\
& = \int_\Omega |f(w)|^p|h(w)|^{ps} \left( \int_\Omega |K(z,w)|^{(1-t)r} |h(z)|^{-rs} dV(z) \right)^{p/r}dV(w)\\
& \lesssim \int_\Omega |f(w)|^p|h(w)|^{ps}|h(w)|^{-ps} dV(w) = ||f||_p^p.
\end{align*}

We must choose $s$ such that \eqref{genschur2} holds with $\epsilon = rs$ and \eqref{genschur1} holds with $\epsilon =qs$. This is possible if and only if $a'<rs<b'$ and $a < qs < b$.
\end{proof}

\begin{proof}[Proof of Theorem 1.2]
Let $K_\eta(z,w) = B_{m/n}(z,w)(|w_2|^{n} - |w_1|^{ m})^\eta(1-|w_2|^2)^\eta$ be the kernel of $T_\eta$. To apply Lemma \ref{schurlemma2} we will verify the estimate
\begin{align}
\int_{\mathbb{H}_{m/n}} |K_\eta(z,w)|^{\rho} |h(w)|^{-\epsilon} dV(w) &\lesssim |h(z)|^{-\epsilon} \label{est}
\end{align}
for suitable $\rho \le \frac{2}{2-\eta}$ and $\epsilon$ to be determined. For this, we claim that $|K_\eta(z,w)|^\rho$ satisfies a type-$(\rho(A-2n) +2n, \rho A)$ estimate, where $A = 2n -1 +\frac{1-n}{m}$ . To see this, note that

\begin{align*}
    |K_{\eta}(z,w)|^{\rho} &\lesssim \frac{|z_2\bar{w}_2|^{\rho A}}{|1-z_2\bar{w}_2|^{2 \rho}|z_2^n\bar{w}_2^n -z_1^m\bar{w}_1^m|^{2 \rho} } (|w_2|^{n} - |w_1|^{m})^{\rho\eta}(1-|w_2|^2)^{\rho\eta}\\
    & = \frac{|z_2|^{\rho A - 2n(\rho - 1)}|\bar{w}_2|^{\rho A}}{|1-z_2\bar{w}_2|^{2}|z_2^n\bar{w}_2^n -z_1^m\bar{w}_1^m|^{2} } \left( \frac{(1-|w_2|^2)^{\rho\eta}}{|1-z_2\bar{w}_2|^{2 (\rho-1)}}\cdot \frac{|z_2|^{2n(\rho-1)} (|w_2|^{n} - |w_1|^{m})^{\rho\eta} }{|z_2^n\bar{w}_2^n -z_1^m\bar{w}_1^m|^{2 (\rho-1)}}\right).
\end{align*}
We claim that the term in the parenthesis is uniformly bounded on $\mathbb{H}_{m/n} \times \mathbb{H}_{m/n}.$ Examining the denominator, we see 
\[|z_2^n\bar{w}_2^n -z_1^m\bar{w}_1^m| \ge (|z_2 \bar{w}_2|^n - |z_1 \bar{w}_1|^m) \ge |z_2|^{n}(|w_2|^n - |w_1|^m).\]
Since $\eta \ge \frac{2(\rho -1)}{\rho}$, $(|w_2|^n -|w_1|^m)^{\rho \eta} \le (|w_2|^n -|w_1|^m)^{2(\rho -1)}$ when $|w_2|^n -|w_1|^m <1$. Therefore
\[\frac{|z_2|^{2n(\rho-1)} (|w_2|^{n} - |w_1|^{m})^{\rho\eta} }{|z_2^n\bar{w}_2^n -z_1^m\bar{w}_1^m|^{2 (\rho-1)}} \le \frac{|z_2|^{2n(\rho-1)} (|w_2|^{n} - |w_1|^{m})^{\rho\eta}}{|z_2|^{2n(\rho -1)}(|w_2|^n - |w_1|^m)^{2(\rho -1)}} \lesssim 1. \]
For the other factor, we have 
$1-|w_2|^2 = (1+|w_2|)(1-|w_2|)
< 2(1-|w_2|) < 2(1 -|z_2||w_2|).$ Hence,
\[ \frac{(1-|w_2|^2)^{\rho \eta }}{|1-z_2 \bar{w}_2|^{2(\rho -1)}} \le \frac{(1-|w_2|^2)^{2(\rho -1)}}{|1-z_2 \bar{w}_2|^{2(\rho -1)}} < 2.\]

This finishes the proof of the claim, and of the estimate
\begin{align}
    |K_{\eta}(z,w)|^{\rho} \lesssim  \frac{|z_2|^{\rho A - 2n(\rho - 1)}|\bar{w}_2|^{\rho A}}{|1-z_2\bar{w}_2|^{2}|z_2^n\bar{w}_2^n -z_1^m\bar{w}_1^m|^{2} }\label{typecd}.
\end{align}

Now the proof of Proposition \ref{cd} shows \eqref{est} holds as well as the analogous statement when the roles of $z$ and $w$ are interchanged. Thus we can apply Lemma \ref{schurlemma2}. We wish to maximize $r$ subject to the constraints
\begin{align}
    tq \le \rho  \quad \text{ and } \quad (1-t)r \le \rho \label{constraint1}
\end{align}
where $\rho$ is fixed. Solving this yields a maximum $r = p+G(\rho)$, where 
\[G(\rho) = \frac{p^2}{\frac{\rho}{\rho-1} -p}.\]
However there is an extra constraint 
\begin{align}
    \dfrac{a'}{b} < \dfrac{r}{q} < \dfrac{b'}{a} \label{r_inequality}
\end{align}
imposed by the range of $\epsilon$. We find $a$, $a'$, $b$, and $b'$ from \eqref{typecd} in the same way as in Proposition \ref{cd}, i.e.
\[a = \frac{\rho(2n - A)}{2n}, \qquad b = \frac{1}{2n}\left(\rho A +\frac{2n}{m}-2n+2\right) \]
\[a' = \frac{2n-\rho A}{2n}, \qquad b' = \frac{1}{2n} \left( \rho(A-2n) +\frac{2n}{m} +2\right).\]
Some algebra shows the right inequality of \eqref{r_inequality} is equivalent to
\begin{align}
    r &<  q\left(\frac{1}{\rho} \cdot \frac{2n+2m}{m+n-1} - 1 \right).\label{constraint2}
\end{align}
Note that \eqref{constraint2} is stronger than \eqref{constraint1} if and only if $ p+G(\rho) \ge M := \frac{2n+2m}{m+n-1}$. Indeed, $p+G(\rho) > M$ implies 
\[ \rho  > \frac{Mp}{M(p-1) +p}. \]
This implies that
\begin{align*}
    q \left( \frac{M}{\rho} -1\right) &< q \left( \frac{M(p-1) +p}{p} -1\right)\\
    &= Mq \cdot \frac{p-1}{p} = M.
\end{align*}
Hence if $p+G(\eta)< M$ the maximum $r$ is $p+G(\eta)$, and if $p+G(\eta) \ge M$ there is no restriction except $r<M.$ Solving these inequalities for $\eta$ shows the upper bound for the proof. \\

To prove the additional boundedness results for small $p$, we have that the left inequality of \eqref{r_inequality} is equivalent to 
\begin{equation*}
    q\left(\frac{2m+2n}{\rho(2mn+1-m-n)-2mn+2m+2n}-1\right) < r.
\end{equation*}
Then, since $m\geq1$ and $n\geq1$, we have that $4m^2n+4mn^2+2mn\eta-2m^2-2m-2n^2-2n\geq0$ with equality only when $m=n=1$ and $\eta=0$. Thus, we have that
\begin{align*}
    0&\leq \frac{4m^2n+4mn^2+2mn\eta-2m^2-2m-2n^2-2n}{(2mn+1-m-n)(2m+2n+2+mn\eta)}\\
    &< \frac{(L+q)2mn-L(2m+2n)}{(L+q)(2mn+1-m-n)}.\\
\end{align*}
Therefore, if $\frac{2mn-2n-2m}{2mn+1-m-n}<\rho<\frac{(L+q)2mn-L(2m+2n)}{(L+q)(2mn+1-m-n)}$ then
\begin{align*}
    \frac{L+q}{q(2m+2n)}&<\frac{1}{\rho(2mn+1-m-n)-2mn+2m+2n}.\\
\end{align*}
Rearranging, we obtain
\begin{align*}
    L&<q\left(\frac{2m+2n}{\rho(2mn+1-m-n)-2mn+2m+2n}-1\right).\\
\end{align*}
Thus, we have that $r>L$ for all $\eta$. However, by the conditions on Lemma \ref{schurlemma2}, $r>1$. Thus, if $\eta\geq\frac{m+n-1}{mn}$, $r>1$ is the only bounding on $r$ and we have proved (a) and (b) of Theorem \ref{ToeplitzN}.\\

We prove unboundedness in the same way as in Theorem 1.1. Let $f(z) = z_1^{\beta_1}\bar{z}_2^{\beta_2}$ be as before, where $\beta_1$ and $\beta_2 = - \ell( \beta_1)$ are defined by \eqref{leastexp}. Then we have
\begin{align*}
    \left|\left|T^*_\eta (z_1^{\beta_1}\overline{z_2}^{\beta_2})\right|\right|_p^p
    &\approx \int_{\mathbb{H}_{m/n}}\left|z_1^{\beta_1}z_2^{-\beta_2}(|z_2|^n - |z_1|^m)^\eta(1-|z_2|^2)^\eta\right|^pdV(z)\\
    &\approx \int_0^1r_2^{-p\beta_2+1+np\eta}(1-r_2^2)^{p\eta}\int_0^{r_2^{n/m}}r_1^{p\beta_1+1}\left(1 - \frac{r_1^m}{r_2^n}\right)^{p\eta} dr_1dr_2\\
    &\gtrsim \int_0^1r_2^{-p\beta_2+1+np\eta}\left[1-p\eta r_2^2\right]\int_0^{r_2^{n/m}}r_1^{p\beta_1+1} - p\eta r_1^{p\beta_1 + 1+ m}r_2^{-n} dr_1dr_2\\
    &= \int_0^1r_2^{-p\beta_2+1+np\eta + \frac{np\beta_1}{m}+\frac{2n}{m}}-p\eta r_2^{-p\beta_2+3+np\eta +     \frac{np\beta_1}{m}+\frac{2n}{m}}dr_2
\end{align*}
where we used the approximation $(1-x)^r = 1-rx+O(x^2)$ for $|x|<1.$
The last integral diverges when
\begin{equation}
    -1\geq -p\beta_2+1+np\eta + \frac{np\beta_1}{m}+\frac{2n}{m}. \label{1.6Sharpness2}
\end{equation}
Plugging in $\beta_2 = -\ell(\beta_1) = 1+\frac{n(\beta_1+1)-1}{m}$ we see \eqref{1.6Sharpness2} becomes
    \[p \ge \frac{2m+2n}{m+n-1-mn\eta}.\]
We conclude that $T_{\eta}^*(f)\notin L^p(\mathbb{H}_{m/n})$ for $\eta < \frac{m+n-1}{mn}$ and $p \geq \frac{2m+2n}{m+n-1-mn\eta}$.\\

The same calculation with $\eta = 0$ shows that $||T_\eta (f)||_p < \infty$ if and only if $p <\frac{2m+2n}{m+n-1}$. Thus, we have shown that $T_\eta$ is unbounded on $L^p(\mathbb{H}_{m/n})$ for $p\in\left[1,\frac{2m+2n}{m+n+1+mn\eta}\right]\cup\left[\frac{2m+2n}{m+n-1},\infty\right)$ for $\eta<\frac{m+n-1}{mn}$ and $T_\eta$ is unbounded on $L^p(\mathbb{H}_{m/n})$ for $p\in\left[\frac{2m+2n}{m+n-1},\infty\right).$
\end{proof}
\begin{remark}
If $\eta \ge 2$ then one can show $T_\eta$ maps $L^1(\mathbb{H}_{m/n})$ to $A^1(\mathbb{H}_{m/n})$ boundedly by showing $T^*_\eta$ is bounded on $L^\infty(\mathbb{H}_{m/n})$. This follows from estimating the kernel of $T_\eta^*(f)$  as in the above proof:
\[|T_\eta^*f(z)| \lesssim  ||f||_\infty \int_{\mathbb{H}_{m/n}}|z_2|^A|w_2|^{A-2n}dV(w)\]
where $A = 2n-1+\frac{1-n}{m}$. Since the integral is finite this proves the claim.
\end{remark}
\section{Further Questions}

In this section, we mention a few possible further  research directions. We plan to address some of these in subsequent projects.

\begin{enumerate}[1.]
    \item Another positive radial symbol of interest is the Bergman kernel on the diagonal $\phi(z) = B(z,z)^{-\eta}$. \cite{khanh} considers Toeplitz operators with these symbols on $\mathbb{H}_k$ for integers $k >0$, and show that a similar smoothing effect as in Theorem 2.2 occurs. They also previously studied the analogous symbols on general psuedoconvex domains in \cite{khanh0}. Since $B_{m/n}(z,z)^{-\eta}$ is a power of the distance to the boundary with extra factors of $|z_1|$ and $|z_2|$, the techniques in this paper might also be used to generalize their result to all rational Hartgos triangles.
    \item There are higher dimensional generalizations of Hartogs triangle that have been studied in the literature. \cite{beberok, hartoggen} consider the domain $\{(z,w) \in \C^n \times \C : ||z|| < |w|^k < 1 \}$ where $||z||$ is the Euclidean norm on $\C^n$. The Bergman projection and Toeplitz operators as above exhibit similar restricted $L^p$ regularity. \cite{chen} considers a different generalization of Hartogs triangle defined using biholomorphic maps. One can also consider domains of the form $\{(z_1,\dots,z_n) \in \C^n: |z_1|^{\gamma_1} < |z_2|^{\gamma_2} < \dots < |z_n|^{\gamma_n}< 1\}$ where $\gamma_1, \dots, \gamma_{n} >0$. This domain has the same kind of boundary singularities as $\mathbb{H}_\gamma$, and it would be interesting to see how the $L^p$ mapping properties of the Bergman projection and Toeplitz operators depend on the exponents and their arithmetic properties.
    \item In addition to $L^p$ estimates, one can investigate similar questions in Sobolev norms or one can investigate compactness of Toeplitz operators. In particular, it is not clear if there is a characterization of the compactness of $T_{\phi}$ on $L^2(\mathbb{H}_{\gamma})$ in terms of the Berezin transform of the symbol $\phi$ as in \cite{axzheng}.
    One can show that if $T_\phi$ is compact on $L^2(\mathbb{H}_{m/n})$ then the Berezin transform vanishes on the boundary. We do not know if the converse holds, but we believe it is likely that there is a counterexample.
\end{enumerate}
\medskip

\section*{Acknowledgements}
This project was completed as a part of the Polymath REU 2020; the authors thank the organizers and the participants of the program. This research was partially sponsored by NSF grant DMS-1659203, NSF grant DGE-1745038, and a grant from the NSA.

%\bibliographystyle{alpha}
%\bibliography{ref}

\end{document}